\documentclass[a4paper]{amsart}
\usepackage{graphicx}
\usepackage{amssymb}
\usepackage{amsmath}
\usepackage{amsthm,amsfonts,bbm}
\usepackage{amscd}
\usepackage[all,2cell]{xy}

\UseAllTwocells \SilentMatrices

\newtheorem{thm}{Theorem}[section]

\newtheorem{cor}[thm]{Corollary}
\newtheorem{lem}[thm]{Lemma}

\newtheorem{prop-def}[thm]{Proposition-Definition}
\theoremstyle{definition}
\newtheorem{defi}[thm]{Definition}
\theoremstyle{remark}

\numberwithin{equation}{section}
\numberwithin{figure}{section}

\def\A{\mathcal{A}}

\def\B{\mathcal{B}}
\def\C{\mathbb{C}}

\def \Dg{\mathfrak{D}}

\def \i{\mathbf{i}}
\def\I{\mathcal{I}}

\def\In{\Gamma}

\def\la{\lambda}

\newcommand{\lr}[1]{\langle#1\rangle}

\def\PS{\mathbb{S}}
\def \PV{\mathbb{V}}

\def \S{\mathcal{S}}

\def \Spec{\mbox{\rm Spec}}

\def \Tr{\mbox{Tr}}
\def\u{{\mathbf u}}

\def\x{{\mathbf x}}

\def\Z{\mathbb{Z}}

\begin{document}
\title[Stabilizing index and cyclic index of hypergraphs]
{The stabilizing index and cyclic index of coalescence and Cartesian product of uniform hypergraphs}

\author[Y.-Z. Fan]{Yi-Zheng Fan$^*$}
\address{School of Mathematical Sciences, Anhui University, Hefei 230601, P. R. China}
\email{fanyz@ahu.edu.cn}
\thanks{$^*$The corresponding author.
This work was supported by National Natural Science Foundation of China (Grant No. 11871073).}

\author[M.-Y. Tian]{Meng-Yu Tian}
\address{School of Mathematical Sciences, Anhui University, Hefei 230601, P. R. China}
\email{tianmy@stu.ahu.edu.cn}

\author[M. Li]{Min Li}
\address{School of Mathematical Sciences, Anhui University, Hefei 230601, P. R. China}
\email{lim@stu.ahu.edu.cn}



\begin{abstract}
Let $G$ be connected uniform hypergraph and let $\A(G)$ be the adjacency tensor of $G$.
The stabilizing index of $G$ is the number of eigenvectors of $\A(G)$ associated with the spectral radius, and
the cyclic index of $G$ is the number of eigenvalues of $\A(G)$ with modulus equal to the spectral radius.
Let $G_1 \odot G_2$ and $G_1 \Box G_2$ be the coalescence and Cartesian product of connected $m$-uniform hypergraphs $G_1$ and $G_2$ respectively.
In this paper, we give explicit formulas for the the stabilizing indices and cyclic indices of  $G_1 \odot G_2$ and $G_1 \Box G_2$ in terms of those of $G_1$ and $G_2$ or the invariant divisors of their incidence matrices over $\Z_m$, respectively.
\end{abstract}

\maketitle

\section{Introduction}
A \emph{tensor} $\A=(a_{i_{1} i_2 \ldots i_{m}})$ of order $m$ and dimension $n$ over a field $\mathbb{F}$ refers to a
 multiarray of entries $a_{i_{1}i_2\ldots i_{m}}\in \mathbb{F}$ for all $i_{j}\in [n]:=\{1,2,\ldots,n\}$ and $j\in [m]$, which can be viewed to be the coordinates of the classical tensor (as a multilinear function) under an orthonormal basis.
A \emph{hypergraph} $G=(V,E)$ consists of a vertex set $V=\{v_1,v_2,{\cdots},v_n\}$ denoted by $V(G)$ and an edge set $E=\{e_1,e_2,{\cdots},e_k\}$ denoted by $E(G)$,
 where $e_i \subseteq V$ for $i \in [k]:=\{1,2,\cdots,k\}$.
 If $|e_i|=m$ for each $i \in [k]$ and $m \geq2$, then $G$ is called an \emph{$m$-uniform} hypergraph.

 Lim \cite{Lim} and Qi \cite{Qi} introduced the eigenvalues of tensors independently in 2005.
In 2012 Cooper and Dutle\cite{CD} introduced the adjacency tensor of uniform hypergraphs, and use the eigenvalues of the tensor to
characterize the structure of hypergraphs.
Recently the spectral hypergraph theory has been an active topic in graph and hypergraph theory \cite{BL,FBH, FHB, FanTPL, LM, Ni1, Ni, PZ, SQH, ZKSB, zhou}.

By the Perron-Frobenius theorem of  nonnegative tensors \cite{CPZ1,FGH,YY1,YY2,YY3},
for an irreducible or weakly irreducible nonnegative tensor $\A$ of order $m$,
the spectral radius $\rho(\A)$ is an eigenvalue of $\A$ associated with a unique positive eigenvector up to a scalar, called the \emph{Perron vector} of $\A$.
If $m \ge 3$, $\A$ can have more than one eigenvector associated with $\rho(\A)$, which is different from the case of matrices (of order $m=2$).
They are two problems related to the spectral radius:
(P1) the number of eigenvectors of $\A$ associated with $\rho(\A)$, and (P2) the number of eigenvalues of $\A$ with modulus equal to $\rho(\A)$.

For the problem (P1), Fan et al. \cite{FBH} introduced the stabilizing index $s(\A)$ for a general tensor $\A$ of order $m$, and
showed that $s(\A)$ is exactly the answer to (P1) if $\A$ is nonnegative weakly irreducible.
If $\A$ is further symmetric, then $s(\A)$ can be obtained explicitly by the Smith normal form of the incidence matrix of $\A$ over $\Z_m$.

Formally, let $\A$ be a tensor of order $m$ and dimension $n$.
Define
\begin{equation}\label{D0}
\Dg^{(0)}(\A)=\{D: D^{-(m-1)}\A D=\A, d_{11}=1\},
\end{equation}
where $D$ is an $n \times n$ invertible diagonal matrix normalized such that $d_{11}=1$, and the product here is defined in \cite{Shao}.
It was proved $\Dg^{(0)}(\A)$ is an abelian group under the usual matrix multiplication, and is only determined by the support
or the zero-nonzero pattern of $\A$ (\cite[Lemmas 2.5-2.6]{FBH}).

\begin{defi}[\cite{FBH}]
For a general tensor $\A$, the cardinality of the abelian group $\Dg^{(0)}(\A)$, denoted by $s(\A)$, is called the \emph{stabilizing index} of $\A$.
\end{defi}

For the problem (P2), if $\A$ is nonnegative weakly irreducible,
the number of eigenvalues of $\A$ with modulus equal to $\rho(\A)$ is called the \emph{cyclic index} of $\A$ by Chang et al. \cite{CPZ2}, denoted by $c(\A)$.
By Perron-Frobenius theroem, $\Spec(\A)=e^{\i \frac{2\pi}{c(\A)}}\Spec(\A)$ in this case, where $\Spec(\A)$ denotes the spectrum of $\A$.
So $c(\A)$ implies the spectral symmetry of tensors.
Fan et al. \cite{FHB} defined the spectral symmetry for a general tensor.

\begin{defi}[\cite{FHB}]\label{ell-sym}
Let $\A$ be a general tensor, and let $\ell$ be a positive integer.
The tensor $\A$ is called {\it spectral $\ell$-symmetric} if
\begin{equation}\label{sym-For}
\Spec(\A)=e^{\i \frac{2\pi}{\ell}}\Spec(\A).
\end{equation}
The maximum $\ell$ such that (\ref{sym-For}) holds is the {\it cyclic index} of $\A$,  denoted by $c(\A)$.
\end{defi}

If $\A$ is nonnegative weakly irreducible, then $c(\A)$ in Definition \ref{ell-sym} is consistent with that given in \cite{CPZ2}.
If $\A$ is further symmetric of order $m$, Fan et al. \cite{FHB} characterized the spectral $\ell$-symmetry by using $(m,\ell)$-coloring of $\A$.
In general, it was shown in \cite{FHB} that
$$c(\A)=\hbox{gcd}\{d: \Tr_d(\A) \ne 0\},$$
where $\Tr_d(\A)$ is the generalized $d$-th order trace of $\A$ (\cite{MS,SQH}).
However, it is very difficult to compute the generalized traces of a tensor.

In this paper we will discuss the problems (P1)  and (P2) for the adjacency tenor of
 the coalescence $G_1 \odot G_2$ and Cartesian product $G_1 \Box G_2$ of connected $m$-uniform hypergraphs $G_1$ and $G_2$.
By the above discussion, the answers to (P1) and (P2) in this situation are equivalent to determine the stabilizing index and cyclic index.
We give explicit formulas for the the stabilizing indices and cyclic indices of  $G_1 \odot G_2$ and $G_1 \Box G_2$ in terms of those of $G_1$ and $G_2$ or the invariant divisors of their incidence matrices, respectively,
where the \emph{stabilizing index} and \emph{cyclic index} of a uniform hypergraph $G$ are  referring to its adjacency tensor $\A(G)$, denoted by $s(G)$ and $c(G)$ respectively.

Some notations are used throughout the paper.
Let $S$ be a nonempty set.
Denote by $\I_S$ (respectively, $I_S$, $\mathbf{1}_S$) an identity tensor (respectively, identity matrix, all-one vector) with entries indexed by the elements of $S$.
Sometimes we use $\I, I, \mathbf{1}$ if there exists no confusion.
Denote by $O$ and $\mathbf{0}$ respectively a zero matrix and a zero vector of whose size can be implicated by the context.
Denote $\lr{a,b}:=\gcd(a,b)$.

\section{Preliminaries}
\subsection{Tensors and hypergraphs}
Let $\A=(a_{i_{1} i_2 \ldots i_{m}})$ be a tensor of order $m$ and dimension $n$ over a $\C$.
If all entries $a_{i_1i_2\cdots i_m}$ of $\A$ are invariant under any permutation of its indices, then $\A$ is called a \emph{symmetric tensor}.
The irreducibility or weakly irreducibility of a tensor can be referred to \cite{CPZ1,FGH}.
Given a vector $\x\in \mathbb{C}^{n}$, $\A \x^{m-1} \in \mathbb{C}^n$, which is defined as follows:
\begin{align*}
(\A \x^{m-1})_i & =\sum_{i_{2},\ldots,i_{m}\in [n]}a_{ii_{2}\ldots i_{m}}x_{i_{2}}\cdots x_{i_m}, i \in [n].
\end{align*}
A tensor $\mathcal{I}=(i_{i_1i_2\ldots i_m})$ of order $m$ and dimension $n$ is called an \emph{identity tensor}, if $i_{i_{1}i_2 \ldots i_{m}}=1$ for
   $i_{1}=i_2=\cdots=i_{m} \in [n]$ and $i_{i_{1}i_2 \ldots i_{m}}=0$ otherwise.

\begin{defi}[\cite{Lim,Qi}] Let $\A$ be an $m$-th order $n$-dimensional tensor.
For some $\lambda \in \mathbb{C}$, if the polynomial system $(\lambda \mathcal{I}-\A)x^{m-1}=0$,
or equivalently $\A x^{m-1}=\lambda x^{[m-1]}$, has a solution $x\in \mathbb{C}^{n}\backslash \{0\}$,
then $\lambda $ is called an \emph{eigenvalue} of $\A$ and $x$ is an \emph{eigenvector} of $\A$ associated with $\lambda$,
where $x^{[m-1]}:=(x_1^{m-1}, x_2^{m-1},\ldots,x_n^{m-1})$.
\end{defi}

The \emph{characteristic polynomial} $\varphi_\A(\la)$ of $\A$ is defined as the resultant of the polynomials $(\la \I-\A)\x^{m-1}$ (see \cite{Qi,CPZ2,Ha}).
It is known that $\la$ is an eigenvalue of $\A$ if and only if it is a root of $\varphi_\A(\la)$.
The \emph{spectrum} of $\A$, denoted by $\Spec(\A)$, is the multi-set of the roots of $\varphi_\A(\la)$.
The largest modulus of the elements in $\Spec(\A)$ is called the \emph{spectral radius} of $\A$, denoted by $\rho(\A)$.
Let $\mathbb{P}^{n-1}$ be the complex projective spaces of dimension $n-1$, and let $\la$ be an eigenvalue of $\A$.
Consider the projective variety
$$\PV_\la=\PV_\la(\A)=\{\x \in \mathbb{P}^{n-1}: \A\x^{m-1}=\la \x^{[m-1]}\}.$$
which is called the \emph{projective eigenvariety} of $\A$ associated with $\la$ \cite{FBH}.
In this paper the \emph{number of eigenvectors} of $\A$ is considered in $\PV_\la(\A)$, i.e.
the eigenvectors differing by a scalar is counted once as the same eigenvector.

Let $G=(V,E)$ be a hypergraph.
$G$ is called \emph{nontrivial} if it contains more than one vertex.
A {\it walk} $W$ of length $t$ in $G$ is a sequence of alternate vertices and edges: $v_{0}e_{1}v_{1}e_{2}\ldots e_{t}v_{t}$,
    where $v_{i} \ne v_{i+1}$ and $\{v_{i},v_{i+1}\}\subseteq e_{i}$ for $i=0,1,\ldots,t-1$.
If $v_0=v_l$, then $W$ is called a {\it circuit}.
A circuit of $G$ is called a {\it cycle} if no vertices or edges are repeated except $v_0=v_l$.
The hypergraph $G$ is said to be {\it connected} if every two vertices are connected by a walk;
 and is called a {\it hypertree} if $G$ is connected and acyclic.

Let $G_1$, $G_2$ be two vertex-disjoint connected nontrivial hypergraphs, and let $v_1\in V(G_1), v_2\in V(G_2)$.
The \emph{coalescence} of $G_1$, $G_2$ with respect to $v_1, v_2$,
   denoted by $G:=G_1(v_1)\odot G_2(v_2)$, is obtained from $G_1$, $G_2$ by identifying $v_1$ with $v_2$ and forming a new vertex $u$,
which is also written as $G_1(u)\odot G_2(u)$.
In this case, $u$ is called a \emph{cut vertex} of $G$, and $G_1, G_2$ are called a \emph{branches} of $G$.
A \emph{block} of $G$ is the (edge) maximal connected sub-hypergraph of $G$ without cut vertices.
A \emph{pendent block} of $G$ is a block of $G$ sharing exactly one vertex with other blocks.
Surely, if $G$ contains cut vertices, it must contain pendent blocks.

\begin{defi}\cite{Im, Gri, CD}
Let $G$ and $H$ be two $m$-uniform hypergraphs.
 The \emph{Cartesian product} of $G$ and $H$, denoted by $G\Box H$, has vertex set $V(G\Box H)=V(G)\times V(H)$, where $\{(i_1,j_1),\cdots,(i_m,j_m)\}\in E(G\Box H)$ if and only if one of the following two conditions holds:
(1) $i_1=\cdots=i_m$ and $\{j_1,\cdots,j_m\}\in E(H)$, (2) $j_1=\cdots=j_m$ and $\{i_1,\cdots,i_m\}\in E(G)$.
\end{defi}

Let $G$ be an $m$-uniform hypergraph on $n$ vertices $v_1,v_2,\ldots,v_n$.
The {\it adjacency tensor} of $G$ is defined as $\mathcal{A}(G)=(a_{i_{1}i_{2}\ldots i_{k}})$, an order $m$ dimensional $n$ tensor, where
\[a_{i_{1}i_{2}\ldots i_{m}}=\left\{
 \begin{array}{ll}
\frac{1}{(m-1)!}, &  \mbox{if~} \{v_{i_{1}},v_{i_{2}},\ldots,v_{i_{m}}\} \in E(G);\\
  0, & \mbox{otherwise}.
  \end{array}\right.
\]
Observe that $\A(G)$ is nonnegative and symmetric, and it is weakly irreducible if and only if $G$ is connected \cite{PZ,YY3}.
In this paper the \emph{spectrum}, \emph{eigenvalues} and \emph{spectral radius} of $G$ are referring to its adjacency tensor $\A(G)$.
The spectral radius of $G$ is denoted by $\rho(G)$.

It is proved in \cite{Shao} that the adjacency tensor of $G \Box H$ is
\begin{equation}\label{dir-ten}
\A(G\Box H)=\A(G) \otimes \I_{V(H)} + \I_{V(G)} \otimes \A(H).
\end{equation}
The hypergraph $G \Box H$ is connected if and only if $G$ and $H$ are both connected \cite{Im,Gri}.

\subsection{Stabilizing index}

Let $\A$ be a symmetric tensor of order $m$ and dimension $n$.
Set $$E(\A)=\{(i_1, i_2, \cdots, i_m)\in [n]^m: a_{i_1i_2\cdots i_m}\neq 0, 1\le i_1\le \cdots \le i_m \le n\}.$$
Define
\[\gamma_{e,j}=|\{k: i_k=j, e=(i_1, i_2, \cdots, i_m) \in E(\A), k \in [m]\}|\]
and obtain an $|E(\A)|\times n$ matrix $\In(\A)=(\gamma_{e,j})$, called the \emph{incidence matrix} of $\A$ \cite{FBH}.
The \emph{incidence matrix} of $G$, denoted by $\In(G)=(\gamma_{e,v})$, coincides with that of $\A(G)$, that is $\gamma_{e,v}=1$ if $v \in e$, and $\gamma_{e,v}=0$ otherwise.

For a matrix $B \in \Z_m^{k \times n}$,
there exist invertible matrices $P \in \Z_m^{k \times k}$ and $Q \in \Z_m^{n \times n}$ such that
\begin{equation} \label{smith}
PBQ=\left(
\begin{array}{ccccccc}
d_1 & 0 & 0 &  & \cdots & & 0\\
0 & d_2 & 0 &  & \cdots & &0\\
0 & 0 & \ddots &  &  & & 0\\
\vdots &  &  & d_r &  & & \vdots\\
 & & & & 0 & & \\
  & & & &  & \ddots & \\
0 &  &  & \cdots &  & &0
\end{array}
\right),
\end{equation}
where $r \ge 0$, $ 1 \le d_i \le m-1$, $d_i | d_{i+1}$ for $i \in [r-1]$, and $d_i |m$ for all $i \in [r]$.
The matrix in (\ref{smith}) is called the {\it Smith normal form} of $B$ over $\Z_m$,
where $d_1, \ldots, d_r$ are the \emph{invariant divisors} of $B$ over $\Z_m$.
We call two matrices $B, C$ are \emph{equivalent} over $\Z_m$ if $B=PCQ$ for some invertible matrices $P, Q$ over $\Z_m$.

Let $\A$ be a nonnegative weakly irreducible tensor of order $m$ and dimension $n$.
By Lemma 2.5 of \cite{FBH}, there is a bijection between $\PV_{\rho(\A)}$ and $\Dg^{(0)}(\A)$, and hence $|\PV_{\rho(\A)}|=s(\A)$,
namely $s(\A)$ is the number of eigenvectors of $\A$ associated with $\rho(\A)$.
By assigning a quasi-Hadamard product $\circ$ in $\PV_{\rho(\A)}$,
by Lemma 3.1 of \cite{FBH}, $(\PV_{\rho(\A)},\circ)$ is an abelian group isomorphism to $(\Dg^{(0)}(\A), \cdot)$.
Assume $\A$ is further symmetric.
By Lemma 2.5 of \cite{FBH}, $D^m=I$ for each $D \in \Dg^{(0)}(\A)$.
Then $(\PV_{\rho(\A)}, \circ)$ and $\Dg^{(0)}(\A)$ both admit $\Z_m$-modules and are isomorphic to each other.
Define the following $\Z_m$-module:
$$\PS_0(\A)=\{\x \in \Z_m^n: \In(\A)\x=\mathbf{0} \hbox{~over~} \Z_m, x_1=0\}.$$

\begin{thm}[\cite{FBH}, Lemma 3.3, Theorem 3.4, Theorem 3.6] \label{stru}
Let $\A$ be a symmetric weakly irreducible nonnegative tensor of order $m$ and dimension $n$.
Suppose that the incidence matrix $\In(\A)$ has a Smith normal form over $\Z_m$ as in (\ref{smith}).
Then $1 \le r \le n-1$, and as $\Z_m$-modules,
$$\PV_{\rho(\A)} \cong \Dg^{(0)}(\A) \cong \PS_0(\A) \cong \oplus_{i, d_i \ne 1} \Z_{d_i} \oplus (n-1-r)\Z_m.$$
\end{thm}

By applying Theorem \ref{stru} to the adjacency tensor, we have the following result.

\begin{cor}\label{vrhoG}
Let $G$ be a connected $m$-uniform hypergraph on $n$ vertices.
Suppose that the incidence matrix $\In(G)$ has a Smith normal form over $\Z_m$ as in (\ref{smith}).
Then $1 \le r \le n-1$, and as $\Z_m$-modules
\begin{equation}\label{struEQ}
\PV_{\rho(G)} \cong \oplus_{i, d_i \ne 1} \Z_{d_i} \oplus (n-1-r)\Z_m, ~~~|\PV_{\rho(G)}|=s(G)=m^{n-1-r}\Pi_{i=1}^r d_i.
\end{equation}
\end{cor}

\subsection{Cyclic index}
We mainly introduce some knowledge on the cyclic index of symmetric tensors and uniform hypergraphs.

\begin{lem}\cite[Lemma 2.7, Lemma 3.2]{FHB} \label{1-cyc}
Let $\A$ be a tensor of order $m$.
If $\A$ is spectral $\ell$-symmetric, then $\ell | c(\A)$.
If $\A$ is further symmetric, then $\ell | m$, and hence $c(\A) |m$.
\end{lem}

\begin{defi}\label{spe-ell-sym}\cite{FHB}
Let $m \ge 2$ and $\ell \ge 2$ be integers such that $ \ell \mid  m$.
An $m$-th order $n$-dimensional tensor $\A$ is called \emph{$(m,\ell)$-colorable}
if there exists a map $\phi:[n] \to [m]$ such that if $a_{i_1\ldots i_m} \ne 0$, then
$ \phi(i_1)+\cdots+\phi(i_m) \equiv \frac{m}{\ell} \mod m.$
Such $\phi$ is called an $(m,\ell)$-coloring of $\A$.
\end{defi}

\begin{thm}\cite{FHB}\label{ml-color}
Let $\A$ be a symmetric weakly irreducible nonnegative tensor of order $m$.
Then $\A$ is spectral $\ell$-symmetric if and only if $\A$ is $(m,\ell)$-colorable.
\end{thm}

A uniform hypergraph $G$ is called \emph{spectral $\ell$-symmetric} if $\A(G)$ is spectral $\ell$-symmetric.

\begin{defi}\label{spe-ell-sym-graph}
Let $m \ge 2$ and $\ell \ge 2$ be integers such that $ \ell \mid  m$.
An $m$-uniform hypergraph $G$ on $n$ vertices is called \emph{$(m,\ell)$-colorable}
if there exists a map $\phi: V(G) \to [m]$ such that if $\{v_{i_1},\ldots, v_{i_1}\} \in E(G)$, then
\begin{equation}\label{gen-col} \phi(v_{i_1})+\cdots+\phi(v_{i_1}) \equiv \frac{m}{\ell} \mod m.\end{equation}
Such $\phi$ is called an \emph{$(m,\ell)$-coloring} of $G$.
\end{defi}

\begin{cor}\cite{FHB}\label{ml-color-G}
Let $G$ be a connected $m$-uniform hypergraph.
Then $G$ is spectral $\ell$-symmetric if and only if $G$ is $(m,\ell)$-colorable.
\end{cor}

Now Eq. (\ref{gen-col}) is equivalent to
\[
\In(G) \phi =\frac{m}{\ell} \mathbf{1} \hbox{~over~} \Z_m,
\]
where $\phi=(\phi(v_1),\ldots,\phi(v_n))^\top$ is considered as a column vector.
So, Corollary \ref{ml-color-G} can be rewritten as follows.

\begin{cor} \label{sym-Zm}
Let $G$ be a connected $m$-uniform hypergraph.
Then $G$ is spectral $\ell$-symmetric if and only if the equation
\begin{equation}\label{ell-Zm}
\In(G) \x =\frac{m}{\ell} \mathbf{1} \hbox{\rm~over~} \Z_m
\end{equation} has a solution.
\end{cor}

So the cyclic index $c(G)$ is the maximum divisor $\ell$ of $m$ such that $G$ is $(m,\ell)$-colorable.

\section{Stabilizing index of hypergraphs}
We first discuss the stabilizing index of the coalescence of connected uniform hypergraphs.

\begin{lem}\label{cyc-2bran}
Let $G = G_1(u) \odot G_2(u)$, where $G_1, G_2$ are both nontrivial connected $m$-uniform hypergraphs.
Then $ s(G) =s(G_1) \cdot s(G_2)$.
\end{lem}

\begin{proof}
Suppose that $G_1=(V_1, E_1)$ and $G_2=(V_2,E_2)$, where $|V_1|=n_1$ and $|V_2|=n_2$.
Then $G$ has $n_1+n_2-1$ vertices.
Write the incident matrix $\In(G)$ of $G$ with rows and columns labeled as follows:
\begin{equation}\label{inc-G}
\In(G)=\bordermatrix{
    & {\small V_1\backslash \{u\}} & u & {\small V_2 \backslash \{u\}} \cr
{\small E_1} & B_1 & \u_1 & O \cr
{\small E_2} & O & \u_2 & B_2
},
\end{equation}
where
    $\In(G_1)=(B_1, \u_1)$ and $\In(G_2)=(\u_2, B_{G_2})$.

As the sum of each row of $\In(G_1)$ is $m =0 \mod \Z_m$, adding all columns of $\In(G)$ indexed by $V_1\backslash \{u\}$ to the column indexed by $u$,
$\In(G)$ is equivalent to $\In^{(1)}$ in Eq. (\ref{inc-G1}), and $\In(G_1)$ is equivalent to $(B_1, \mathbf{0})$.
Similarly, adding all columns of $\In^{(1)}$ indexed by $V_2\backslash \{u\}$ to the column indexed by $u$, $\In^{(1)}$ is equivalent to
$\In^{(2)}$ in Eq. (\ref{inc-G1}), and $\In(G_2)$ is equivalent to $(\mathbf{0},B_2)$.
\begin{equation}\label{inc-G1}
\In^{(1)}=\begin{pmatrix}
 B_1 & \mathbf{0} & O \\
 O & \u_2 & B_2
\end{pmatrix},~~~
\In^{(2)}=\begin{pmatrix}
 B_1 & \mathbf{0} & O \\
 O & \mathbf{0} & B_2
\end{pmatrix}.
\end{equation}

Assume that $\In(G_1)$ has invariant divisors $d_1, \ldots, d_s$, and $\In(G_2)$ has invariant divisors $\bar{d}_1, \ldots, \bar{d}_t$, where $1 \le s \le n_1-1$ and $1 \le t \le n_2-1$.
As $\In(G_1)$ is equivalent to $(B_1, \mathbf{0})$, $B_1$ has invariants $d_1, \ldots, d_s$.
Similarly, $B_2$ has invariants $\bar{d}_1, \ldots, \bar{d}_t$.
So, $\In^{(2)}$, and hence $\In(G)$ has invariant divisors $d_1, \ldots, d_s, \bar{d}_1, \ldots, \bar{d}_t$.
By Corollary \ref{vrhoG},
$$ s(G)=m^{(n_1+n_2-1)-1-(s+t)} \prod_{i \in [s]} d_i \prod_{j \in [t]} \bar{d}_j=s(G_1) \cdot s(G_2).$$
\end{proof}

\begin{thm}\label{cyc-sblock}
Let $G$ be a connected $m$-uniform hypergraph with cut vertices.
If $G$ has $s$ blocks $G_1, \ldots, G_s$, where $s \ge 2$, then
 $s(G)=\Pi_{i=1}^s s(G_i)$.
\end{thm}

\begin{proof}
Observer that $G$ contains a pendent block, say $G_1$.
Write $G=G_1 \odot H$, where $H$ has $s-1$ blocks.
By Lemma \ref{cyc-2bran}, $s(G)=s(G_1) \cdot s(H)$.
The result follows by induction on $s$.
\end{proof}

\begin{cor}\label{cyc-tree}
Let $T$ be an $m$-uniform hypertree with $s$ edges.
Then $s(T)=m^{s(m-2)}$.
\end{cor}

\begin{proof}
By the definition, each edge $e$ of $T$ is a block.
So $T$ has $s$ blocks.
For a hypergraph consisting of only one edge $e$, by Example 4.3 of \cite{FBH}, $s(e)=m^{m-2}$.
The result follows by Theorem \ref{cyc-sblock}.
\end{proof}

In Corollary \ref{cyc-tree}, when the number of edges (or vertices) of $T$ goes to infinity,
then $s(T)$, the number of eigenvectors of $\A(T)$ associated with $\rho(T)$, goes to infinity.
We now give an example to show the opposite situation.
Let $K_n^{[m]}$ be an $m$-uniform complete hypergraph on $n \ge m+1$ vertices.
 It was shown that $s(K_n^{[m]})=1$ by Example 4.4 of \cite{FHB}.
 If taking a specified vertex $u$ of  $K_n^{[m]}$, we construct the coalescence of $n$ copies of $K_n^{[m]}$ as follows:
 $ n \odot K_n^{[m]}:=K_n^{[m]}(u) \odot K_n^{[m]}(u) \odot \cdots \odot K_n^{[m]}(u)$.
 By Theorem \ref{cyc-sblock}, $s(n \odot K_n^{[m]})=1$.

Next we discuss the stabilizing index of the Cartesian product of two  connected uniform hypergraphs.
We need the notion of direct product (or Kronecker product) of two tensors of order $m$,
which generalizes the direct product of matrices (of order $m=2$).

\begin{defi}\cite{Shao}
Let $\A,\B$ be two tensors of order $m$  with dimension $n_1,n_2$, respectively.
The \emph{direct product} $\A \otimes \B$ is defined to be a tensor of order $m$ and dimension $n_1n_2$, whose entries are
\[ (\A \otimes \B)_{(i_1,j_1),\ldots, (i_m,j_m)}=a_{i_1 \ldots i_m}b_{j_1 \ldots j_m}, i_t \in [n_1], j_t \in [n_2], t \in [m].
\]
\end{defi}

\begin{lem}\label{inc-cart}
Let $G_1,G_2$ be two $m$-uniform hypergraphs.
Then under a labeling of the vertices and edges, the Cartesian product $G_1 \Box H$ has the incidence matrix
\[
\In({G_1\Box G_2})=\left(\begin{array}{c}
I_{V(G_1)}\otimes \In({G_2})\\
\In({G_1})\otimes I_{V(G_2)}\\
\end{array}\right).
\]
\end{lem}

\begin{proof}
By the definition, we have a bipartition $\{E_1,E_2\}$ of the edge set $E(G_1 \Box G_2)$ as follows:
\begin{align}
E_1&:=\cup_{u \in V(G_1)} \{u\}\times E(G_2)=\cup_{u \in V(G_1)}\{ \{u\}\times f : f \in E(G_2)\}.\label{2nd-edge}\\
E_2&:=\cup_{v \in V(G_2)} E(G_1)\times \{v\}=\cup_{v \in V(G_2)}\{ e \times \{v\}: e \in E(G_1)\},\label{1st-edge}
\end{align}
For each edge $\{u\}\times f \in E_1$, and each vertex $(u',v) \in V(G_1) \times V(G_2)$,
$$ \In({G_1\Box G_2})_{\{u\}\times f, (u', v)}=\delta_{uu'}\cdot \In(G_2)_{fv}=(I_{V(G_1)}\otimes \In({G_2}))_{(u,f),(u',v)},$$
where $\delta_{uu'}$ is the Kronecker delta symbol, i.e. $\delta_{uu'}=1$ if $u=u'$,  and  $\delta_{uu'}=0$ otherwise.
For each edge $e \times \{v\} \in E_2$, and each vertex $(u,v') \in V(G_1) \times V(G_2)$,
$$ \In({G_1\Box G_2})_{e \times \{v\}, (u, v')}=\In(G_1)_{eu} \cdot \delta_{vv'}=(\In({G_1})\otimes I_{V(G_2)})_{(e,v),(u,v')}.$$
The result follows.
\end{proof}

\begin{thm}\label{stab-carts}
Let $G_1$ be a connected $m$-uniform hypergraph on $n_{G_1}$ vertices whose incidence matrix $\In(G_1)$ has invariants $d_1, \ldots d_{r_{G_1}}$,
and let $G_2$ be a connected $m$-uniform hypergraph on $n_{G_2}$ vertices whose incidence matrix $\In(G_2)$ has invariants $\bar{d}_1, \ldots \bar{d}_{r_{G_2}}$.
Then
\begin{align*}
s(G_1\Box G_2)&=m^{(n_{G_1}-r_{G_1})(n_{G_2}-r_{G_2})-1}\prod_{i\in [r_{G_1}],\atop j\in [r_{G_2}]}\lr{d_i,\bar{d}_j}\prod_{i\in [r_{G_1}]}d_i^{n_{G_2}-r_{G_2}}
\prod_{j\in [r_{G_2}]}\bar{d}_j^{n_{G_1}-r_{G_1}}.
\end{align*}
In particular, if $m$ is prime, then
$$s(G_1\Box G_2)=m^{(n_{G_1}-r_{G_1})(n_{G_2}-r_{G_2})-1},$$
where $r_{G_2},r_{G_2}$ are the ranks of $\In(G_1), \In(G_2)$ over the field $\Z_m$ respectively.
\end{thm}

\begin{proof}
There exist invertible matrices $P_1, Q_1$ and $P_2, Q_2$ over $\Z_m$ such that
\begin{equation}\label{sm}
P_1\In(G_1)Q_1=\Lambda_{G_1}, \; P_2\In(G_2)Q_2=\Lambda_{G_2},
\end{equation}
where $\Lambda_{G_1}, \Lambda_{G_2}$ are the Smith normal forms of $\In({G_1}), \In({G_2})$ respectively.
Now applying an invertible transformation to $\In({{G_1}\Box {G_2}})$, by Lemma \ref{inc-cart}, we have
\begin{align*}
\small \left(\begin{array}{cc}
Q^{-1}_1\otimes P_2 & O \\
O & P_1\otimes {Q^{-1}_2}
\end{array}\right)
\left(\begin{array}{c}
I_{V({G_1})}\otimes \In({G_2}) \\
\In({G_1})\otimes I_{V({G_2})}
\end{array}\right)
Q_1\otimes Q_2 =
\left(\begin{array}{c}
I_{V({G_1})}\otimes \Lambda_{G_2} \\
\Lambda_{G_1}\otimes I_{V({G_2})}
\end{array}\right)
:=\hat{\In}.
\end{align*}

Let $n:=n_{G_1} \cdot n_{G_2}$, and let
\[\S:=\{\x \in\Z_m^{n}: \In({{G_1}\Box {G_2}})\x=\mathbf{0} \hbox{~over~} \Z_m\},~~~
\bar{\S}:=\{\x \in \Z_m^{n}:\hat{\In}\x=\mathbf{0}  \hbox{~over~} \Z_m\}.\]
It is easily seen that $\S$ and $\bar{\S}$ are both $\Z_m$-modules,
and $\S \cong \bar{\S}$ as $\In({{G_1}\Box {G_2}})$  is equivalent to $\hat{\In}$.
So it suffices to consider the equation $\hat{\In}\x=\textbf{0}$ over $\mathbb{Z}_m$, which is equivalent to the following two equations:
 \begin{eqnarray}
(I_{V({G_1})}\otimes \Lambda_{G_2}) ~\x  =  \textbf{0}\hbox{~over~} \Z_m ,\label{1st-equ}\\
(\Lambda_{G_1}\otimes I_{V({G_2})})~\x  =  \textbf{0}\hbox{~over~} \Z_m.\label{2nd-equ}
\end{eqnarray}

Let $V(G_1)=\{1,2,\ldots,n_{G_1}\}$ and $V(G_2)=\{1', 2', \ldots, n'_{G_2}\}$.
By the equation (\ref{1st-equ}), for each $i=1,2,\ldots,n_{G_1}$,
\begin{equation}\label{H}
\Lambda_{G_2}\x_{i\cdot}=\textbf{0} \hbox{~over~} \mathbb{Z}_m,
\end{equation}
where $\x_{i\cdot}=(x_{i1'},\cdots,x_{in'_{{G_2}}})^\top$.
Noting that $\Lambda_{G_2}$ has invariant divisors $\bar{d}_1, \ldots \bar{d}_{r_{G_2}}$, so
$$
\bar{d}_jx_{ij'}=0, j=1,2,\ldots,r_{G_2},
$$
that is, for $i=1,2,\ldots,n_{G_1}$,
\begin{equation}\label{1st-equ-1}
x_{ij'}\in\mathbb{Z}_m (m/\bar{d}_j) \cong \mathbb{Z}_{\bar{d}_j}, j=1,2,\ldots,r_{G_2},
\end{equation}
\begin{equation}\label{1st-equ-2}
x_{ij'}\in\mathbb{Z}_m,  j=r_{{G_2}}+1,\ldots,n_{G_2}.
\end{equation}

By the equation (\ref{2nd-equ}), for each $j=1,2,\cdots,n_{G_2}$,
\begin{equation}\label{G}
\Lambda_{G_1}\x_{\cdot j'}=\textbf{0} \hbox{~over~} \Z_m,
\end{equation}
where $\x_{\cdot j'}=(x_{1j'},\cdots,x_{n_{G_1} j'})^\top$.
Noting that $\Lambda_{G_1}$ has invariant divisors $d_1, \ldots d_{r_{G_1}}$, so
$$d_{i}x_{ij'}=0 \mod \Z_m, i=1,2,\cdots,r_{G_1},$$
that is, for  $j=1,2,\ldots,n_{G_2}$,
\begin{equation}\label{2nd-equ-1}
x_{ij'}\in\mathbb{Z}_m (m/d_{i}) \cong\mathbb{Z}_{d_{i}}, i=1,2,\cdots,r_{G_1},
\end{equation}
\begin{equation}\label{2nd-equ-2}
x_{ij'}\in\mathbb{Z}_m, i={r_{G_1}}+1,\cdots,n_{G_1}.
\end{equation}

Combining (\ref{1st-equ-1}-\ref{1st-equ-2}) and (\ref{2nd-equ-1}-\ref{2nd-equ-2}), we have four cases:
\begin{itemize}
\item[(i)] for $i=1,2,\cdots,r_{G_1},~j=1,2,\cdots,r_{G_2}$,
$$x_{ij'}\in \Z_m (m/d_{i}) \cap \Z_m (m/\bar{d}_j) \cong\Z_{d_{i}}\cap \Z_{\bar{d}_{j}}\cong \mathbb{Z}_{\lr{d_{i},\bar{d}_j}};$$

\item[(ii)] for $i=1,2,\cdots,r_{G_1},~j={r_{G_2}}+1,\cdots,n_{G_2}$,
$$x_{ij'}\in \Z_m(m/d_{i}) \cap \Z_m \cong\Z_{d_{i}};$$

\item[(iii)] for $i={r_{G_1}}+1,\cdots,n_{G_1},~j=1,2,\cdots,r_{G_2}$,
$$x_{ij'}\in\Z_m  \cap \Z_m (m/\bar{d}_j) \cong\mathbb{Z}_{\bar{d}_{j}};$$

\item[(iv)] for $i={r_{G_1}}+1,\cdots,n_{G_1},~j={r_{G_2}}+1,\cdots,n_{G_2}$,
$$x_{ij'}\in\Z_m\cap \Z_m=\mathbb{Z}_m.$$

\end{itemize}
So
\begin{align*}
\bar{\S} & \cong \left(\oplus_{i\in [r_{G_1}]\atop j\in [r_{G_2}]}\Z_{\lr{d_{i},\bar{d}_j}}\right)
\oplus \left(\oplus_{i\in [r_{G_1}]}(n_{G_2}-r_{G_2})\Z_{d_{i}}\right)\\
& \; \; \; \; \oplus \left(\oplus_{j\in [r_{G_2}]}(n_{G_1}-r_{G_1})\Z_{\bar{d}_{j}}\right)\oplus (n_{G_1}-r_{G_1})(n_{G_2}-r_{G_2})\Z_m.
\end{align*}
Let $\PS=\{\x \in \Z_m^{n}:\In({G_1\Box G_2})\x=\textbf{0} \hbox{~over~} \Z_m, x_1=0\}$.
Then $\PS \cong \S/(\Z_m \mathbf{1}) \cong \bar{\S}/(\Z_m (Q_1 \otimes Q_2)^{-1} \mathbf{1})$ (or see Theorem 3.6 of \cite{FBH}).
By Theorem \ref{stru}, we have $s(G_1\Box G_2) =\frac{1}{m}|\bar{\S}|$, which equals the number as in the theorem.

If $m$ is a prime integer, then $\Z_m$ is a field.
So the invariant divisors of $\In(G_1),\In(G_2)$ over $\Z_m$ are all ones,
and $r_{G_1},r_{G_2}$ are the ranks of $\In(G_1),\In(G_2)$ over $\Z_m$ respectively.
The result follows.
\end{proof}

\section{Cyclic index of hypergraphs}
We first discuss the cyclic index of the coalescence of connected uniform hypergraphs.

\begin{lem}\label{stab-cyc-2bran}
Let $G = G_1(u) \odot G_2(u)$, where $G_1, G_2$ are both nontrivial connected $m$-uniform hypergraphs.
Then $ c(G) =\lr{c(G_1), c(G_2)}$.
\end{lem}

\begin{proof}
As $G$ is spectral $c(G)$-symmetric, by Corollary \ref{ml-color-G}, $G$ is $(m,c(G))$-colorable, and has an $(m,c(G))$-coloring $\Phi$.
Noting that $G_1$ and $G_2$ have no common edges, so $\Phi|_{V(G_1)}$, the restriction of $\Phi$ on the vertices of $G_1$, is an $(m,c(G))$-coloring of $G_1$.
So $G_1$ is $(m,c(G))$-colorable and hence is spectral $c(G)$-symmetric by Corollary \ref{ml-color-G}, implying that $c(G) \mid c(G_1)$ by Lemma \ref{1-cyc}.
Similarly, $\Phi|_{V(G_2)}$ is an $(m,c(G))$-coloring of $G_2$, and $c(G) \mid c(G_2)$.
So $c(G) \mid \lr{c(G_1), c(G_2)}$.

Note that $G_1$ is spectral $c(G_1)$-symmetric and $G_2$ is spectral $c(G_2)$-symmetric.
So by Corollary \ref{ml-color-G}, $G_1$ has an $(m,c(G_1))$-coloring $\Phi_{G_1}$, and $G_2$ has a $(m,c(G_2))$-coloring $\Phi_{G_2}$, both satisfying Eq. (\ref{gen-col}).
Observe that for any $t \in \Z_m$, $\Phi_{G_1}+t \mathbf{1}$ is still an $(m,c(G_1))$-coloring of $G_1$ as $mt=0 \mod \Z_m$.
So we can assume $\Phi_{G_1}(u)=0$ and similarly $\Phi_{G_2}(u)=0$.
Define a coloring $\Phi: V(G) \to [m]$ such that
$$\Phi(v)=\left\{
 \begin{array}{cl}
\frac{c(G_1)}{\lr{c(G_1), c(G_2)}}\Phi_{G_1}(v) &  \mbox{if~} v \in V(G_1) \backslash \{u\};\\
0, & \mbox{if~} v=u;\\
\frac{c(G_2)}{\lr{c(G_1), c(G_2)}}\Phi_{G_2}(v) &  \mbox{if~} v \in V(G_2) \backslash \{u\}.
 \end{array}\right.
$$
Then $\Phi|_{G_1}=\frac{c(G_1)}{\lr{c(G_1), c(G_2)}}\Phi_{G_1}$ and $\Phi|_{G_2}=\frac{c(G_2)}{\lr{c(G_1), c(G_2)}}\Phi_{G_2}$.
For each edge $e=\{v_{i_1}, \ldots, v_{i_m}\} \in E(G)$, if $e \in E(G_1)$, by  Eq. (\ref{gen-col}),
\begin{eqnarray*}
 \Phi(v_{i_1})+\cdots+\Phi(v_{i_m})& = & \frac{c(G_1)}{\lr{c(G_1), c(G_2)}}(\Phi_{G_1}(v_{i_1})+\cdots+\Phi_{G_1}(v_{i_m}))\\
 & = & \frac{c(G_1)}{\lr{c(G_1), c(G_2)}}\frac{m}{c(G_1)} \mod \Z_m\\
 & = & \frac{m}{\lr{c(G_1), c(G_2)}} \mod \Z_m.
 \end{eqnarray*}
Similarly, if $e \in E(G_2)$, also by  Eq. (\ref{gen-col}), we have
$$ \Phi(v_{i_1})+\cdots+\Phi(v_{i_m})=\frac{m}{\lr{c(G_1), c(G_2)}} \mod \Z_m.$$
So $G$ has an $(m,\lr{c(G_1), c(G_2)})$-coloring, and hence is spectral $\lr{c(G_1), c(G_2)}$-symmetric, implying $\lr{c(G_1), c(G_2)} \mid c(G)$.

The result follows by the above discussion.
\end{proof}

\begin{thm}\label{stab-cyc-sblock}
Let $G$ be a connected $m$-uniform hypergraph with cut vertices.
If $G$ has $s$ blocks $G_1, \ldots, G_s$, where $s \ge 2$, then
 $c(G)=\gcd\{c(G_i): i \in [s]\}$.
\end{thm}

\begin{proof}
Observer that $G$ contains a pendent block, say $G_1$.
Write $G=G_1 \odot H$, where $H$ has $s-1$ blocks.
By Lemma \ref{stab-cyc-2bran}, $c(G)=\gcd(c(G_1), c(H))$.
The result follows by induction on $s$.
\end{proof}

\begin{cor}\label{cyc-tree2}
Let $T$ be an $m$-uniform hypertree with $s$ edges.
Then $c(T)=m$.
\end{cor}

\begin{proof}
By the definition, each edge $e$ of $T$ is a block.
So $T$ has $s$ blocks.
Considering $e$ as an $m$-uniform hypergraph, $e$ is spectral $m$-symmetric by Theorem 3.2 of \cite{SQH}.
So $c(e)=m$ by Lemma \ref{1-cyc}.
The  result follows by Theorem \ref{stab-cyc-sblock}.
\end{proof}

Finally we discuss the cyclic index of the Cartesian product of  connected uniform hypergraphs.

\begin{thm}\label{cyc-cart-2}
Let $G_1,G_2$ be two connected $m$-uniform hypergraphs.
Then
$$c(G_1\Box G_2)=\lr{c(G_1),c(G_2)}.$$
\end{thm}

\begin{proof}
Let $c(G_1\Box G_2)=:c$.
 By Corollary \ref{sym-Zm}, there exists a map $\Phi: V(G_1 \Box G_2) \rightarrow [m]$, such that
\begin{equation}\label{equ-cartP}
\In({G_1\Box G_2})\Phi=\frac{m}{c}\mathbf{1} \hbox{~over~} \Z_m.
\end{equation}
By the definition of Cartesian product, the edge set of  $G_1\Box G_2$ has a bipartition $\{E_1,E_2\}$; see (\ref{2nd-edge}) and (\ref{1st-edge}).

We first consider a subset $E(G_1) \times \{v\}$ of $E_2$ in (\ref{1st-edge}), where $v \in V(G_2)$ is a specified vertex.
Restricting (\ref{equ-cartP}) on $E(G_1)\times \{v\}$, we have
\begin{equation}\label{restrict} \In({G_1\Box G_2})[E(G_1)\times \{v\}, V(G_1) \times V(G_2)] \cdot \Phi=\frac{m}{c}\mathbf{1} \hbox{~over~} \Z_m,
\end{equation}
where $A[I_1,I_2]$ denoted the submatrix of $A$ with rows indexed by $I_1$ and columns indexed by $I_2$.
By the definition, if $v' \in V(G_2)$, $v' \ne v$, then
$$ \In({G_1\Box G_2})[E(G_1)\times \{v\}, V(G_1) \times \{v'\}] =0.$$
So, by (\ref{restrict}), we have
\begin{equation}\label{GH2G} \In({G_1\Box G_2})[E(G_1)\times \{v\}, V(G_1) \times \{v\}] \cdot \Phi[V(G_1) \times \{v\}]=\frac{m}{c}\mathbf{1} \hbox{~over~} \Z_m,
\end{equation}
where $\Phi[V(G_1) \times \{v\}]$ denotes the sub-vector of $\Phi$ indexed by $V(G_1) \times \{v\}$.
The Eq. (\ref{GH2G}) is equivalent to
$$ \In(G_1) \cdot \Phi[V(G_1) \times \{v\}]=\frac{m}{c}\mathbf{1} \mod \Z_m.$$
By Corollary \ref{sym-Zm}, $G_1$ is spectral $c$-symmetric, and hence $c | c(G_1)$ by Lemma \ref{1-cyc}.

Similarly, we consider a subset $\{u\}\times E(G_2)$ of $E_1$ in (\ref{2nd-edge}), where $u \in V(G_1)$ is a specified vertex.
Restricting (\ref{equ-cartP}) on $\{u\}\times E(G_2)$, we have
\begin{equation}\label{restrict2} \In({G_1\Box G_2})[\{u\}\times E(G_2), V(G_1) \times V(G_2)] \cdot \Phi=\frac{m}{c}\mathbf{1} \hbox{~over~} \Z_m,
\end{equation}
If $u' \in V(G_1)$, $u' \ne u$, then
$$ \In({G_1\Box G_2})[\{u\}\times E(G_2), \{u'\} \times V(G_2) ] =0.$$
So
\begin{equation}\label{GH2H} \In({G_1\Box G_2})[\{u\}\times E(G_2), \{u\} \times V(G_2)] \cdot \Phi[\{u\} \times V(G_2)]=\frac{m}{c}\mathbf{1} \hbox{~over~} \Z_m.
\end{equation}
Eq. (\ref{GH2H}) is equivalent to
$$ \In(G_2) \cdot \Phi[\{u\} \times V(G_2)]=\frac{m}{c}\mathbf{1} \mod \Z_m.$$
By Corollary \ref{sym-Zm}, $G_2$ is spectral $c$-symmetric, and hence $c | c(G_2)$.

By the above discussion, we have $c(G_1 \Box G_2) |\lr{c(G_1), c(G_2)}$.
Next we will show $G_1 \Box G_2$ is spectral $\lr{c(G_1), c(G_2)}$-symmetric so that $c(G_1 \Box G_2)=\lr{c(G_1), c(G_2)}$.
By the definition of cyclic index and Corollary \ref{sym-Zm}, there exist maps $\Phi_{G_1}, \Phi_{G_2}$  such that
$$ \In(G_1) \Phi_{G_1}=\frac{m}{c(G_1)}\mathbf{1}_{E(G_1)} \mod \Z_m, \In(G_2) \Phi_{G_2}=\frac{m}{c(G_2)}\mathbf{1}_{E(G_2)} \hbox{~over~} \Z_m.$$
Let $c_1:=c(G_1)$ and $c_2:=c(G_2)$. Define
$$\Phi=\mathbf{1}_{V(G_1)}\otimes \frac {c_2}{\lr{c_1,c_2}}\Phi_{G_2}+\frac {c_1}{\lr{c_1,c_2}}\Phi_{G_1}\otimes\mathbf{1}_{V(G_2)}.$$
By Lemma \ref{inc-cart}, noting that $\In({G_1})\cdot\mathbf{1}=m \mathbf{1} =\mathbf{0} \hbox{~over~} \Z_m$ and
$\In({G_2})\cdot\mathbf{1}= \mathbf{0} \hbox{~over~} \Z_m$ similarly, we have
\begin{align*}
\In({G_1\Box G_2})\Phi &=\left[\begin{array}{c}
\In({G_1})\otimes I_{V(G_2)}\\
I_{V(G_1)}\otimes \In({G_2})\\
\end{array}\right]
\left[ \mathbf{1}_{V(G_1)}\otimes \frac {c_2}{\lr{c_1,c_2}}\Phi_{G_2}+\frac {c_1}{\lr{c_1,c_2}}\Phi_{G_1}\otimes\mathbf{1}_{V(G_2)} \right]\\
&=\left[
\begin{array}{c}
(\In({G_1})\cdot\frac{c_1}{\lr{c_1,c_2}}\Phi_{G_1})\otimes(I_{V(G_2)}\cdot\mathbf{1}_{V(G_2)})\\
(I_{V(G_1)}\cdot\mathbf{1}_{V(G_1)})\otimes(\In({G_2})\cdot\frac {c_2}{\lr{c_1,c_2}}\Phi_{G_2})\\
\end{array}\right]\\
&=\left[
\begin{array}{c}
\frac{c_1}{\lr{c_1,c_2}}\cdot\frac{m}{c_1} \cdot \mathbf{1}_{E(G_1)}\otimes\mathbf{1}_{V(G_2)})\\
\mathbf{1}_{V(G_1)}\otimes\frac {c_2}{\lr{c_1,c_2}}\cdot\frac{m}{c_2}\cdot\mathbf{1}_{E(G_2)}
\end{array}\right]
=\left[
\begin{array}{c}
\frac{m}{\lr{c_1,c_2}}\mathbf{1}_{E(G_1)\times V(G_2)}\\
\frac{m}{\lr{c_1,c_2}}\mathbf{1}_{V(G_1)\times E(G_2)}\\
\end{array}\right]\\
&=\frac{m}{\lr{c_1,c_2}}\mathbf{1} \hbox{~over~} \Z_m.
\end{align*}
By Corollary \ref{sym-Zm}, $G_1\Box G_2$ is spectral $\lr{c_1,c_2}$-symmetric.
The result follows.
\end{proof}

The Cartesian product can be defined on finitely many hypergraphs \cite{Im,Gri}.
For $m$-uniform hypergraphs $G_1, \ldots, G_s$,
the \emph{Cartesian product} $G =\Box_{i=1}^s G_i$ is defined as an $m$-uniform hypergraph with vertex set
$V(G) = \times_{i \in [s]} V(G_i)$, and $E(G)$ consists of $m$-subsets $e$ of $V(G)$ such that
$p_j(e) \in E(G_j)$ for exactly one $j \in [s]$ and $|p_i(e)|=1$ for $ i \ne j$, where, for $j \in [s]$,
$p_j: V(G) \to V(G_j)$ is the projection of the Cartesian product of the vertex sets in to $V(G_j)$.
Here, $G_1, \ldots, G_s$ are called the \emph{factors} of $G$.
It is proved that $G =\Box_{i=1}^s G_i$ is connected if and only if all of its factors $G_1,\ldots,G_s$ are connected \cite{Im,Gri}.

\begin{cor}
Let $G_1, \ldots, G_s$ be connected $m$-uniform hypergraphs.
Then
$$ c(\Box_{i=1}^s G_i)=\gcd\{c(G_i): i \in [s]\}.$$
\end{cor}

\begin{proof}
As the Cartesian product is associative, we can write
$$ G:=\Box_{i=1}^s G_i= (\Box_{i=1}^{s-1}G_i) \Box G_s.$$
So, by Theorem  \ref{cyc-cart-2}, $c(G)=\gcd(c(\Box_{i=1}^{s-1}G_i),c(G_s)$.
The result follows by induction on $s$.
\end{proof}


\begin{thebibliography}{99}

\bibitem{BL} S. Bai, L. Lu, A bound on the spectral radius of hypergraphs with $e$ edges, \emph{Linear Algebra Appl.}, 549 (2018), 203-218.


\bibitem{CPZ1} K. C. Chang, K. Pearson, T. Zhang, Perron-Frobenius theorem for nonnegative tensors, \emph{Commu. Math. Sci.}, 6 (2008), 507-520.

\bibitem{CPZ2} K. C. Chang, K. Pearson, T. Zhang, On eigenvalue problems of real symmetric tensors, \emph{J. Math. Anal. Appl.}, 350 (2009), 416-422.

\bibitem{CD} J. Cooper, A. Dutle, Spectra of uniform hypergraphs, \emph{Linear Algebra Appl.}, 436(9)(2012), 3268-3292.


\bibitem{FGH} S. Friedland, S. Gaubert, L. Han, Perron-Frobenius theorem for nonnegative multilinear forms and extensions, \emph{Linear Algebra Appl.}, 438 (2013), 738-749.


\bibitem{FBH} Y.-Z. Fan, Y.-H. Bao, T. Huang, Eigenvariety of nonnegative symmetric weakly irreducible
tensors associated with spectral radius and its application to hypergraphs, \emph{Linear Algebra Appl.}, 564(2019), 72-94.

\bibitem{FHB} Y.-Z. Fan, T. Huang, Y.-H. Bao, C.-L. Zhuan-Sun and Y.-P. Li,
The spectral symmetry of weakly irreducible nonnegative tensors and connected hypergraphs, \emph{Trans. Amer. Math. Soc.}, 372(3)(2019), 2213-2233.

\bibitem{FanTPL} Y.-Z. Fan, Y.-Y. Tan, X.-X. Peng, A.-H. Liu, Maximizing spectral radii of uniform hypergraphs with few
edges, \emph{Discuss. Math. Graph Theory}, 36(4)(2016), 845-856.


\bibitem{Gri} L. Gringmann, M. Hellmuth, P. F. Stadler, The Cartesian product of hypergraphs, \emph{J. Graph Theory}, 70(2) (2012), 180-196.


\bibitem{Ha} R. Hartshorne, \textit{Algebraic Geometry}, Springer-Verlag, New York, 1977.

\bibitem{HY} S. Hu, K. Ye, Mulplicities of tensor eigenvalues, \emph{Commu. Math. Sci.}, 14 (2016), 1049-1071.


\bibitem{Im} W. Imrich,  Kartesisches Produkt von Mengensystemen und Graphen, \emph{Studia Sci. Math. Hungar.}, 2 (1967), 285-290.

\bibitem{Lim}L.-H. Lim, Singular values and eigenvalues of tensors: A variational approach,
in \emph{Computational Advances in Multi-Sensor Adapative Processing}, 2005 1st IEEE International Workshop, IEEE, Piscataway, NJ, 2005, pp. 129-132.

\bibitem{LM}L. Lu, S. Man, Connected hypergraphs with small spectral radius, \emph{Linear Algebra Appl.}, 509 (2016), 206-227.

\bibitem{MS} A. Morozov, Sh. Shakirov, Analogue of the identity Log Det=Trace Log for resultants, \emph{J. Geom. Phys.}, 61 (2010), 708-726.



\bibitem{Ni1} V. Nikiforov, Analytic methods for uniform hypergraphs, \emph{Linear Algebra Appl.}, 457 (2014), 455-535.

\bibitem{Ni} V. Nikiforov, Hypergraphs and hypermatrices with symmetric spectrum,  \emph{Linear Algebra Appl.},  519 (2017), 1-18.

\bibitem{PZ} K. Pearson and T. Zhang, On spectral hypergraph theory of the adjacency tensor, \emph{Graph Combin.}, 30 (5) (2014): 1233-1248.

\bibitem{Shao} J.-Y. Shao, A general product of tensors with applications,  \emph{Linear Algebra Appl.}, 439 (2013), 2350-2366.

\bibitem{SQH} J.-Y. Shao, L. Qi, S. Hu, Some new trace formulas of tensors with applications in spectral hypergraph theory,
 \emph{Linear Multilinear Algebra}, 63(5) (2015), 971-992.

\bibitem{Qi} L. Qi, Eigenvalues of a real supersymmetric tensor, \emph{J. Symbolic Comput.}, 40 (2005), 1302-1324.


\bibitem{YY1}  Y. Yang and Q. Yang, Further results for Perron-Frobenius theorem for nonnegative tensors, \emph{SIAM J Matrix Anal. Appl.}, 31 (5) (2010), 2517-2530.

\bibitem{YY2} Y. Yang and Q. Yang, Further results for Perron-Frobenius theorem for nonnegative tensors II, \emph{SIAM J Matrix Anal. Appl.}, 32 (4) (2011), 1236-1250.

\bibitem{YY3} Y. Yang, Q. Yang, On some properties of nonnegative weakly irreducible tensors, Available at arXiv: 1111.0713v2.


\bibitem{ZKSB} W. Zhang, L. Kang, E. Shan, Y. Bai,  The spectra of uniform hypertrees, \emph{Linear Algebra Appl.}, 533(2017), 84-94.


\bibitem{zhou} J. Zhou, L. Sun, W. Wang, C. Bu, Some spectral properties of uniform hypergraphs, \emph{Electron. J. Combin.}, 21(2014) \#P4.24.
\end{thebibliography}
\end{document}